\UseRawInputEncoding
\documentclass[12pt]{article}
\usepackage[centertags]{amsmath}
\usepackage{amsfonts}
\usepackage{amssymb}
\usepackage{latexsym}
\usepackage{amsthm}
\usepackage{newlfont}
\usepackage{graphicx}
\usepackage{listings}
\usepackage{booktabs}
\usepackage{abstract}
\date{}

\newlength{\defbaselineskip}
\newcommand{\setlinespacing}[1]%
           {\setlength{\baselineskip}{#1 \defbaselineskip}}

\newcommand{\N}{{\mathbb{N}}}

\newcommand{\actaqed}{\hfill $\actabox$}
{\medskip\noindent \textit{Proof of #1. }}%
{\actaqed \medskip}

\def\D{{\mathcal D}}

\def \Tr{\mathcal T}

\def \cE{\mathcal E}

\def\R{{\mathbb R}}

\def \<{\langle}
\def\>{\rangle}

\def \ff{\varphi}
\def\al{\alpha}
\def\bt{\beta}
\def\ga{\gamma}

\def\bt{\beta}

\def \csp{\overline{\operatorname{span}}}

\newtheorem{Theorem}{Theorem}[section]
\newtheorem{Lemma}{Lemma}[section]

\newtheorem{Remark}{Remark}[section]

\numberwithin{equation}{section}

\newcommand{\be}{\begin{equation}}
\newcommand{\ee}{\end{equation}}

\begin{document}

\title{On the rate of convergence of greedy algorithms}
\author{  V.N. Temlyakov\thanks{ Steklov Mathematical Institute of Russian Academy of Sciences, Moscow, Russia; Lomonosov Moscow State University;  Moscow Center of Fundamental and Applied Mathematics;  University of South Carolina.}}
\maketitle
\begin{abstract}
{ We prove some results on the rate of convergence of greedy algorithms, which provide expansions. 
We consider both the case of Hilbert spaces and the more general case of Banach spaces. The new ingredient of the paper is that we bound the error of approximation by the product of both norms -- 
the norm of $f$ and the $A_1$-norm of $f$. Typically, only the $A_1$-norm of $f$ is used. In particular, we establish that some greedy algorithms (Pure Greedy Algorithm (PGA) and its generalizations) are as good as the Orthogonal Greedy Algorithm (OGA) in this new sense of the rate of convergence, while it is known that the PGA is much worth than the OGA in the standard sense. }
\end{abstract}

\section{Introduction}
\label{I}

Let us begin with a general description of the problem. Let $X$ be a Banach space with the norm $\|\cdot\|_X$ and $Y\subset X$ be a subspace of $X$ with a stronger norm $\|f\|_Y\ge \|f\|_X$, $f\in Y$.
Consider a homogeneous approximation operator (linear or nonlinear) $G:Y\to X$, $G(af) = aG(f)$, $f\in Y$, $a\in \R$, and the error of approximation 
$$
e(B_Y,G)_X := \sup_{f\in B_Y}\|f-G(f)\|_X,\quad B_Y:= \{f\,:\, \|f\|_Y \le 1\}. 
$$
Then for any $f\in Y$ we have
\be\label{Y}
\|f-G(f)\|_X \le e(B_Y,G)_X \|f\|_Y.
\ee
The characteristic $e(B_Y,G)_X$ plays an important role in approximation theory with many 
classical examples of spaces $X$ and $Y$, for instance, $X=L_p$ and $Y$ is one the smoothness spaces like Sobolev, Nikol'skii, or Besov space. 

In this paper we focus of the following version of the inequality (\ref{Y}): Find the best $\ga(\al,G,X,Y)$ such that the inequality 
\be\label{XY}
\|f-G(f)\|_X \le \ga(\al,G,X,Y)\|f\|_X^{1-\al} \|f\|_Y^\al,\quad \al \in [0,1],
\ee
holds for all $f\in Y$. Clearly, $\ga(1,G,X,Y)=e(B_Y,G)_X$. Also, it is clear that under assumption 
$\|f-G(f)\|_X \le \|f\|_X$, $f\in Y$, we obtain the trivial bound
$$
\ga(\al,G,X,Y) \le e(B_Y,G)_X^\al.
$$
In this paper we discuss greedy approximation with respect to a given dictionary and prove some 
nontrivial inequalities for $\ga(\al,G,X,Y)$ both in the case of $X$ being a Hilbert space and $X$ being a Banach space. In particular, we establish that some greedy algorithms (Pure Greedy Algorithm (PGA) and its generalizations) are as good as the Orthogonal Greedy Algorithm (OGA) in the sense of inequality (\ref{XY}), while it is known that the the PGA is much worth than the OGA in the sense of the inequality (\ref{Y}) (for definitions and precise formulations see below). 

Let $H$ be a real Hilbert space with the inner product $\<\cdot,\cdot\>$ and norm $\|\cdot\|$.  We say that a set of elements (functions) $\D$ from $H$ is a dictionary (symmetric dictionary) if each $g\in \D$ has norm   one ($\|g\|= 1$), and $\csp \D =H$. In addition we assume for convenience the property of symmetry:
$$
g\in \D \quad \text{implies} \quad -g \in \D.
$$

We define the Pure Greedy Algorithm (PGA). We describe this algorithm for a general dictionary $\D$. If $f\in H$,
we let $g(f)\in \D$ be an element from $\D$ which maximizes $\< f,g\>$. We assume for simplicity that such a maximizer exists; if not suitable modifications are necessary (see Weak Greedy Algorithm below) in the algorithm that follows. We define
$$
G(f,\D):= \<f,g(f)\>g(f)\quad \text{and}\quad R(f,\D) := f-G(f,\D).
$$
 
 {\bf Pure Greedy Algorithm (PGA).} We define $f_0:= f$ and $G_0(f,\D) := 0$. Then, for each $m\ge 1$, we inductively define
$$
G_m(f,\D):= G_{m-1}(f,\D) +G(f_{m-1},\D)
$$
$$
f_m:= f-G_m(f,\D) = R(f_{m-1},\D).
$$
Note that for a given element $f$ the sequence $\{G_m(f,\D)\}$ may not be unique. 

This algorithm is well studied from the point of view of convergence and rate of convergence. The reader 
can find the corresponding results and historical comments in \cite{VTbook}, Ch.2. In this paper we focus 
on the rate of convergence. Typically, in approximation theory we define the rate of convergence for specific classes. In classical approximation theory these are smoothness classes. Clearly, in the general 
setting with arbitrary $H$ and $\D$ we do not have a concept of smoothness similar to the classical 
smoothness of functions. It turns out that the geometrically defined class, namely, the closure of the convex hull of $\D$, which we denote by $A_1(\D)$, is a very natural class. For each $f\in H$ we associate the following norm
$$
\|f\|_{A_1(\D)} := \inf \{M:\, f/M\in A_1(\D)\}.
$$
Clearly, $\|f\|\le \|f\|_{A_1(\D)}$. Then the problem of the rate of convergence of the PGA can be formulated as follows (see \cite{VTbook}, p.95). Find the order of decay of the sequence
$$
\gamma_m(H):=\sup_{\{G_m(f,\D)\},f,\D} \frac{\|f-G_m(f,\D)\|}{ \|f\|_{A_1(\D)}},
$$
where the supremum is taken over all possible choices of $\{G_m(f,\D)\}$, over all elements
$f\in H$, $f\neq 0$, $\|f\|_{A_1(\D)}<\infty$, and over  all dictionaries $\D$. This problem is a central theoretical problem in greedy approximation in Hilbert spaces and it is still open. We mention some of known results here and refer the reader for the detailed history of the problem to \cite{VTbook}, Ch.2. 
It is clear that for any  $f\in H$, such that $\|f\|_{A_1(\D)}<\infty$ we have
$$
\|f-G_m(f,\D)\| \le \gamma_m(H) \|f\|_{A_1(\D)}.
$$
In this paper we discuss the following extension of the asymptotic characteristic $\gamma_m(H)$: For $\al \in (0,1]$ define
$$
\gamma_m(\al,H):=\sup_{\{G_m(f,\D)\},f,\D} \frac{\|f-G_m(f,\D)\|}{\|f\|^{1-\al} \|f\|_{A_1(\D)}^\al}.
$$
Clearly, 
\be\label{I1}
\gamma_m(1,H)= \gamma_m(H),\quad \gamma_m(\al,H)\ge \gamma_m(\bt,H)\quad\text{if}\quad \al\le \bt.
\ee
The first upper bound on $\ga_m(H)$ was obtained in \cite{DT} 
$$
\ga_m(H) \le m^{-1/6}.
$$
Actually, the proof in \cite{DT} (see also \cite{VTbook}, pp.92-93) gives
$$
\ga_m(1/3,H) \le m^{-1/6}.
$$
We establish here the following bounds 
\be\label{pg}
\frac{1}{2}m^{-\al/2} \le \ga_m(\al,H) \le m^{-\al/2}, \quad \al \le 1/3.
\ee
Also in Section \ref{H} we find the right behavior of the asymptotic characteristic similar to $\ga_m(\al,H)$
for a more general algorithm than the PGA, namely, for the Weak Greedy Algorithm with parameter $b$. 

It is interesting to compare the rates of convergence of the PGA and the Orthogonal Greedy Algorithm (OGA). We now give a brief definition of the OGA. We define $f_0^o := f$, $G^o_0(f,\D)=0$ and for 
$m\ge 1$ we inductively define $G^o_m(f,\D)$ to be the orthogonal projection of $f$ onto 
the span of $g(f^o_0)$, ..., $g(f^o_{m-1})$ and set $f^o_m := f-G_m^o(f,\D)$. The analogs of the characteristics $\ga_m(H)$ and $\ga_m(\al,H)$ for the OGA denote by $\ga^o_m(H)$ and $\ga^o_m(\al,H)$. The following bound is proved in \cite{DT} (see also \cite{VTbook}, p.93)
\be\label{I2}
\ga^o_m(H) \le m^{-1/2}.
\ee
It is known (see \cite{Liv}) that $\ga_m(H)$  decays slower than $m^{-0.1898}$. Therefore, from the point of view of the characteristics $\ga_m(H)$ and $\ga^o_m(H)$ the OGA is much better than the PGA. 
We establish here the following bounds 
\be\label{og}
\frac{1}{2}m^{-\al/2} \le \ga^o_m(\al,H) \le m^{-\al/2}, \quad \al \le 1.
\ee
This means that from the point of view of the characteristics $\ga_m(\al,H)$ and $\ga^o_m(\al,H)$ the OGA is the same as PGA for $\al\le 1/3$. It is a very surprising fact. 

We do not know if the upper bound in (\ref{pg}) holds for $\al>1/3$. However, the inequality in (\ref{I1}) and the lower bound for the $\ga_m(H)$ show that
$$
\gamma_m(\al,H)\ge \gamma_m(H) \ge cm^{-0.1898}.
$$
Therefore, the upper bound in (\ref{pg}) cannot be extended beyond $\al_0:= 0.3796$. 

Section \ref{B} deals with the case of a Banach space $X$. Results for the Banach space case are similar 
to those for Hilbert spaces but are not as sharp as their counterparts.

\section{Hilbert space.The Weak Greedy Algorithm with parameter $b$}
\label{H}

 Let a sequence $\tau = \{t_k\}_{k=1}^\infty$, $0\le t_k \le 1$ and a parameter $b\in (0,1]$ be given.  We define the Weak Greedy Algorithm with parameter $b$.  

{\bf Weak Greedy Algorithm with parameter $b$ (WGA($\tau,b$))} We define $f_0:=f_0^{\tau,b}:=f$. Then for each $m\ge 1$, we inductively define:

1) $\ff_m:=\varphi^{\tau,b}_m \in \D$ is any satisfying 
$$
\<f_{m-1},\varphi_m\> \ge t_m \sup_{g\in \D} \<f_{m-1},g\>;
$$

2) 
$$
f_m :=f_m^{\tau,b}:= f_{m-1} -b\<f_{m-1},\varphi_m\>\varphi_m;
$$

3)
$$
G_m(f,\D):=G^{\tau,b}_m(f,\D) := b\sum_{j=1}^m \<f_{j-1},\varphi_j\>\varphi_j.
$$
 In the case $t_k = t$, $k=1,2,\dots$, we write $t$ in the notation instead of $\tau$. 
 
 We proceed to the rate of convergence. The following Theorem \ref{T4.2} was proved in \cite{VT111}.

\begin{Theorem}\label{T4.2} Let $\D$ be an arbitrary dictionary in $H$. Assume $\tau :=\{t_k\}_{k=1}^\infty$ is a nonincreasing sequence and $b\in(0,1]$. Then for $f \in A_1(\D)$ we have
\be\label{4.1}
\|f -G^{\tau,b}_m(f,\D)\| \le e_m(\tau,b)
\ee
where
\be\label{4.1a}
e_m(\tau,b) := \left(1+b(2-b)\sum_{k=1}^m t^2_k\right)^{-\frac{(2-b)t_m}{2(2+(2-b)t_m)}}. 
\ee
\end{Theorem}

Theorem \ref{T4.2} implies the following inequality for any $f$ and any $\D$
\be\label{4.1b}
\frac{\|f -G^{\tau,b}_m(f,\D)\|}{\|f\|_{A_1(\D)}} \le e_m(\tau,b) .
\ee
We now extend Theorem \ref{T4.2} to provide a bound for  
\be\label{4.1c}
 \ga_m^{t,b}(\al,H):= \sup_{\D}\sup_f\sup_{G^{t,b}_m(f,\D)}\frac{\|f -G^{t,b}_m(f,\D)\|}{\|f\|^{1-\al} \|f\|_{A_1(\D)}^\al}  .
\ee
We prove the following Theorem \ref{T4.2a} in the case $t_k=t$, $k=1,2,\dots$.

\begin{Theorem}\label{T4.2a} For any Hilbert space $H$   we have
\be\label{4.1e}
\ga_m^{t,b}(\al,H) \le (1+mb(2-b) t)^{-\al/2},   
\ee
provided $\al \le \frac{(2-b)t}{(2-b)t+2}$.
\end{Theorem}
\begin{proof}  The proof of this theorem goes along the lines of the proof of Theorem \ref{T4.2} in \cite{VT111}. Let $\D$ be a dictionary in $H$. We introduce some notations:
$$
f_k:= f^{t,b}_k,\quad \varphi_k:=\varphi^{t,b}_k, \quad k=0,1,\dots,
$$
$$
a_m := \|f_m\|^2, \quad y_m := 
\<f_{m-1},\varphi_m\>, \quad
m=1,2,\dots,  
$$
and consider the sequence $\{B_n\}$ defined as follows
$$
B_0 := \|f\|_{A_1(\D)},\quad B_m := B_{m-1} +by_m, \quad m=1,2,\dots .
$$
It is clear that $\|f_n\|_{A_1(\D)}\le B_n$, $n=0,1,\dots$. By Lemma 3.5 from [DT] (see also \cite{VTbook}, p.91, Lemma 2.17) we get
\be\label{4.2}
\sup_{g\in \D} \<f_{m-1},g\> \ge \|f_{m-1}\|^2/B_{m-1}.
\ee
From here and from the equality  
$$
\|f_m\|^2 = \|f_{m-1}\|^2 -b(2-b)\<f_{m-1},\varphi_m\>^2
$$
we obtain the following relations
\be\label{4.3}
a_m = a_{m-1} - b(2-b)y_m^2, 
\ee
\be\label{4.4}
B_m = B_{m-1} +by_m, 
\ee
\be\label{4.5}
y_m \ge ta_{m-1}/B_{m-1}. 
\ee
From (\ref{4.3}) and (\ref{4.5}) we obtain
$$
a_m \le a_{m-1} (1-b(2-b)t^2a_{m-1}B_{m-1}^{-2}).
$$ 
Using that $B_{m-1}\le B_m$, we derive from here
\be\label{aB}
a_mB_m^{-2} \le a_{m-1}B_{m-1}^{-2}(1-b(2-b)t^2a_{m-1}B_{m-1}^{-2}).
\ee
We shall need the
following simple known lemma (see, for example, \cite{VTbook}, p. 91, in case $C_1=C_2$). 

\begin{Lemma}\label{HL1} Let $\{x_m\}_{m=0}^\infty$
be a sequence of non-negative
 numbers satisfying the inequalities
$$
x_0 \le C_1, \quad x_{m+1} \le x_m(1 - x_mC_2) , \quad m = 0,1,2, \dots,\quad C_1,C_2>0 .
$$
Then we have for each $m$
$$
x_m \le (C_1^{-1}+C_2m)^{-1} .
$$
\end{Lemma}
\begin{proof} The proof is by induction on $m$. For $m = 0$ the statement
is true by assumption. We assume $x_m \le (C_1^{-1}+C_2m)^{-1}$ and prove that $x_{m+1} \le
(C_1^{-1}+C_2(m+1))^{-1}$. If $x_{m+1} = 0$ this statement is obvious. Assume therefore
that $x_{m+1} > 0$. Then we have 
$$
x_{m+1}^{-1} \ge x_m^{-1}(1 - x_mC_2)^{-1} \ge x_m^{-1}(1 + x_mC_2) =
x_m^{-1} + C_2 \ge C_1^{-1}+(m+1)C_2 ,
$$
which implies $x_{m+1} \le (C_1^{-1}+C_2(m+1))^{-1}$ . 
\end{proof}

We apply Lemma \ref{HL1} with $x_m:= a_mB_m^{-2}$. Then the inequality $\|f\| \le \|f\|_{A_1(\D)}$ implies that we can take $C_1=1$. We set $C_2= b(2-b)t^2$ and obtain from (\ref{aB}) and Lemma \ref{HL1}
\be\label{4.6}
a_mB_m^{-2} \le (1+ mb(2-b)t^2)^{-1}. 
\ee
Relations (\ref{4.3}) and (\ref{4.5}) imply
\be\label{4.7}
a_m \le a_{m-1}-b(2-b)y_mta_{m-1}/B_{m-1} = a_{m-1}(1-b(2-b)ty_m/B_{m-1}). 
\ee
We now need the following simple inequality: For any $x<1$ and any $a>0$ we have
\be\label{ineq}
(1-x)(1+x/a)^a \le 1.
\ee
Rewriting (\ref{4.4}) in the form
\be\label{4.9}
B_m = B_{m-1}(1+by_m/B_{m-1}) 
\ee
and using the inequality (\ref{ineq}) with $x=b(2-b)ty_m/B_{m-1}$ and $a=(2-b)t$ we get from (\ref{4.7}) and (\ref{4.9}) that
\be\label{4.11}
a_mB_{m}^{(2-b)t} \le a_{m-1}B_{m-1}^{(2-b)t}\le \dots \le \|f\|^2\|f\|_{A_1(\D)}^{(2-b)t}.
\ee
Combining (\ref{4.6}) and (\ref{4.11}) we obtain
$$
a_m^{(2-b)t+2} \le \|f\|^4\|f\|_{A_1(\D)}^{2(2-b)t} (1+mb(2-b)t^2)^{-(2-b)t},
$$
which completes the proof of Theorem \ref{T4.2a} with $\al_0 := \frac{(2-b)t}{(2-b)t+2}$.  The case $\al\le \al_0$ follows from Lemma \ref{HL2} below.
\end{proof}

{\bf Lower bounds.} Let $H$ be an infinite dimensional Hilbert space and $\{e_k\}_{k=1}^\infty$ be an 
orthonormal system in $H$. Suppose that our symmetric dictionary $\D$ consists of $\pm e_k$, $k=1,2,\dots$,
and other elements $g\in \D$ have the property $\<g,e_k\>=0$, $k=1,2,\dots$ . We present an example in the case $t=1$. Let $b\in (0,1]$ and $m$ be given. We consider two cases (I) $b\le 1/4$ and (II) $b\in (1/4,1]$.

{\bf (I).} Set $m':= [2bm]+1$ and 
$$
f=\sum_{k=1}^{m'} e_k.
$$
Then at each iteration the WGA($1,b$) will pick one of the $e_k$, $k\in [1,m']$, with the largest coefficient. 
After $m$ iterations we will get
$$
f_m = \sum_{k=1}^{m'}c_ke_k,\qquad c_k \ge 1-(m/m'+1)b \ge 1/4.
$$
Therefore, we obtain
$$
\|f_m\| \ge (m')^{1/2}/4,\quad \|f\| = (m')^{1/2},\quad \|f\|_{A_1(\D)} \le m'.
$$
Thus, for any $\al \in [0,1]$ we find
\be\label{H2}
\frac{\|f_m\|}{\|f\|^{1-\al}\|f\|_{A_1(\D)}^\al} \ge (m')^{-\al/2}/4.
\ee

{\bf (II).} Set
$$
f=\sum_{k=1}^{2m} e_k.
$$
Then
$$
f_m = \sum_{k=1}^{2m}c_ke_k, \quad c_k=1,\, k\in G, \quad c_k = 1-b,\, k\notin G,\quad |G|=m.
$$
Therefore, we obtain
$$
\|f_m\| \ge (m)^{1/2},\quad \|f\| = (2m)^{1/2},\quad \|f\|_{A_1(\D)} \le 2m.
$$
Thus, for any $\al \in [0,1]$ we find
\be\label{H3}
\frac{\|f_m\|}{\|f\|^{1-\al}\|f\|_{A_1(\D)}^\al} \ge 2^{-1/2}(2m)^{-\al/2}.
\ee

Bounds (\ref{H2}) and (\ref{H3}) show that in the case $t=1$ inequality (\ref{4.1e}) is sharp in the sense 
of dependence on $m$ and $b$, when $m$ goes to $\infty$ and $b$ goes to $0$. 

{\bf Proof of (\ref{og}).} The lower bound in (\ref{og}) follows from (\ref{H3}) because in the case of an 
orthonormal system the algorithms PGA and OGA coincide. We now prove the upper bound. 
\begin{Lemma}\label{HL2} Let $g_m(\al,H)$ denote either $\ga_m(\al,H)$ or $\ga^o_m(\al,H)$. Suppose 
that for some $\bt\in (0,1]$ we have
$$
g_m(\bt,H) \le C\ff(m)^{-\bt/2}.
$$
Then for any $\al \in (0,\bt)$ we have
\be\label{H4}
g_m(\al,H) \le C^{\al/\bt}\ff(m)^{-\al/2}.
\ee
\end{Lemma}
\begin{proof} By the definition of the algorithms PGA and OGA we have $\|f_m\|\le \|f\|$ and 
$\|f^o_m\|\le \|f\|$. For both algorithms PGA and OGA the proof is identical. We will carry it out for the PGA. Our assumption gives for any $f$, any dictionary $\D$, and any realization of $G_m(f,\D)$
$$
\|f_m\| = \|f-G_m(f,\D)\| \le \|f\|^{1-\bt}\|f\|_{A_1(\D)}^\bt C\ff(m)^{-\bt/2}.
$$
Therefore, for any $a\in[0,1]$ we have
\be\label{H5}
\|f_m\| \le \|f\|^{1-a}(\|f\|^{1-\bt}\|f\|_{A_1(\D)}^\bt C\ff(m)^{-\bt/2})^a.
\ee
Choosing $a=\al/\bt$, we obtain (\ref{H4}) from (\ref{H5}).
\end{proof}

The upper bound in (\ref{og}) follows from (\ref{I2}) and Lemma \ref{HL2}.
 
\section{Greedy expansions in Banach spaces} 
\label{B}

In this section we extend the results from Section \ref{H} to the case of a Banach space instead of a Hilbert space. We begin with some definitions. Let $X$ be a real Banach space with norm $\|\cdot\|$. As above we say that a set of elements (functions) $\D$ from $X$ is a dictionary (symmetric dictionary) if each $g\in \D$ has norm   one ($\|g\|= 1$), and $\csp \D =X$. In addition we assume for convenience that
the dictionary is symmetric
$$
g\in \D \quad \text{implies} \quad -g \in \D.
$$
We study in this paper   greedy algorithms with regard to $\D$ that provide greedy expansions. 
For a nonzero element $f\in X$ we denote by $F_f$ a norming (peak) functional for $f$: 
$$
\|F_f\| =1,\qquad F_f(f) =\|f\|.
$$
The existence of such a functional is guaranteed by Hahn-Banach theorem. 
Denote 
$$
r_\D(f) := \sup_{F_f}\sup_{g\in \D}F_f(g).
$$
We note that in general a norming functional $F_f$ is not unique. This is why we take $\sup_{F_f}$ over all norming functionals of $f$ in the definition of $r_\D(f)$. It is known that in the case of uniformly smooth Banach spaces (our primary object here) the norming functional $F_f$ is unique. In such a case we do not need $\sup_{F_f}$ in the definition of $r_\D(f)$.

 We consider here approximation in uniformly smooth Banach spaces. For a Banach space $X$ we define the modulus of smoothness
$$
\rho(u) := \rho(u,X):= \sup_{\|x\|=\|y\|=1}(\frac{1}{2}(\|x+uy\|+\|x-uy\|)-1).
$$
A uniformly smooth Banach space is  one with the property
$$
\lim_{u\to 0}\rho(u)/u =0.
$$
It is well known (see for instance \cite{DGDS}, Lemma B.1) that in the case $X=L_p$, 
$1\le p < \infty$ we have
\be\label{1.3}
\rho(u,L_p) \le \begin{cases} u^p/p & \text{if}\quad 1\le p\le 2 ,\\
(p-1)u^2/2 & \text{if}\quad 2\le p<\infty. \end{cases} 
\ee

  We now give a definition of the DGA$(\tau,b,\mu)$, $\tau =\{t_k\}_{k=1}^\infty$, $t_k \in (0,1]$ introduced in \cite{VT111} (see also \cite{VTbook}, Ch.6).  
  
{\bf Dual Greedy Algorithm with parameters $(\tau,b,\mu)$ (DGA$(\tau,b,\mu)$).}
Let $X$ be a uniformly smooth Banach space with the modulus of smoothness $\rho(u)$ and let $\mu(u)$ be a majorant of $\rho(u)$: $\rho(u)\le\mu(u)$, $u\in[0,\infty)$. For a sequence $\tau =\{t_k\}_{k=1}^\infty$, $t_k \in (0,1]$ and a parameter $b\in (0,1]$ we define sequences
$\{f_m\}_{m=0}^\infty$, $\{\ff_m\}_{m=1}^\infty$, $\{c_m\}_{m=1}^\infty$, and $\{G_m\}_{m=0}^\infty$ inductively. Let $f_0:=f$ and $G_0:=0$. If for $m\ge 1$ $f_{m-1}=0$ then we set $f_j=0$ for $j\ge m$ and stop. If $f_{m-1}\neq 0$ then we conduct the following three steps:

1) take any $\ff_m \in \D$ such that
\be\label{3.1}
F_{f_{m-1}}(\ff_m) \ge t_mr_\D(f_{m-1}); 
\ee

2) choose $c_m>0$ from the equation
\be\label{3.2}
\|f_{m-1}\|\mu(c_m/\|f_{m-1}\|) = \frac{t_mb}{2}c_mr_\D(f_{m-1}); 
\ee

3) define
\be\label{3.3}
f_m:=f_{m-1}-c_m\ff_m,\qquad G_m:=G_m^{\tau,b,\mu}:= G_{m-1}+c_m\ff_m. 
\ee

Along with the algorithm DGA$(\tau,b,\mu)$ we consider a slight modification of it, when at step 2) we find 
$c_m$ from the  equation (see \cite{VT111}, Remark 3.1)
\be\label{3.2a}
\|f_{m-1}\|\mu(c_m/\|f_{m-1}\|) = \frac{b}{2}c_mF_{f_{m-1}}(\ff_m). 
\ee
We denote this modification by DGA$(\tau,b,\mu)^*$.

We proceed to studying the rate of convergence of the DGA$(\tau,b,\mu)$ in the uniformly smooth Banach spaces with the power type majorant of the modulus of smoothness: $\rho(u)\le \mu(u)= \ga u^q$, $1<q\le 2$. The following Theorem \ref{T3.1} is from \cite{VT111} (see also \cite{VTbook}, p.372). 

\begin{Theorem}\label{T3.1} Let $\tau :=\{t_k\}_{k=1}^\infty$ be a nonincreasing sequence $1\ge t_1\ge t_2 \dots >0$ and $b\in (0,1)$. Assume that $X$ has a modulus of smoothness $\rho(u)\le \ga u^q$, $q\in (1,2]$. Denote $\mu(u) = \ga u^q$. Then for any dictionary $\D$ and any $f\in A_1(\D)$ the rate of convergence of the DGA$(\tau,b,\mu)$ is given by 
$$
\|f_m\|\le C(b,\ga,q)\left(1+\sum_{k=1}^mt_k^p\right)^{-\frac{t_m(1-b)}{p(1+t_m(1-b))}}, \quad p:= \frac{q}{q-1}.
$$
\end{Theorem}

\begin{Remark}\label{BR1} It is pointed out in \cite{VT111}, Remark 3.2, that Theorem \ref{T3.1} holds 
for the algorithm DGA$(\tau,b,\mu)^*$ as well.
\end{Remark}

Theorem \ref{T3.1} is an analog of Theorem \ref{T4.2}. We now prove an analog of Theorem \ref{T4.2a}. 
%Let $\{f_m\}_{m=0}^\infty$ be the sequence of residuals of the algorithm DGA($t,b,\mu$).
We extend Theorem \ref{T3.1} to provide a bound for  
\be\label{B1}
 \ga_m^{t,b,\mu}(\al,X):= \sup_{\D}\sup_f\sup_{G^{t,b,\mu}_m(f,\D)}\frac{\|f -G^{t,b,\mu}_m(f,\D)\|}{\|f\|^{1-\al} \|f\|_{A_1(\D)}^\al}  .
\ee
The corresponding characteristic for the algorithm DGA$(\tau,b,\mu)^*$ is denoted by $\ga_m^{t,b,\mu}(\al,X)^*$.
We prove the following Theorem \ref{BT1} in the case $t_k=t$, $k=1,2,\dots$.

\begin{Theorem}\label{BT1} For any Banach space $X$ with modulus of smoothness $\rho(u,X) \le \ga u^q$, $1<q\le 2$,  $p:= \frac{q}{q-1}$, we have
\be\label{B2}
\ga_m^{t,b,\mu}(\al,X) \le (1+mc t^p)^{-\al/p}, \quad  c := (1-b)\left(\frac{b}{2\ga}\right)^{\frac{1}{q-1}}, 
\ee
provided $\al \le \frac{t(1-b)}{1+t(1-b)}$. The same inequality holds for the $\ga_m^{t,b,\mu}(\al,X)^*$.
\end{Theorem}

\begin{proof} The proof is identical for both characteristics $\ga_m^{t,b,\mu}(\al,X)$ and $\ga_m^{t,b,\mu}(\al,X)^*$. We carry it out for the $\ga_m^{t,b,\mu}(\al,X)$. From the definition of the modulus of smoothness we have
\be\label{2.11}
\|f_{n-1}-c_n\ff_n\|+\|f_{n-1}+c_n\ff_n\| \le 2\|f_{n-1}\|(1+\rho(c_n/\|f_{n-1}\|)). 
\ee
Using the definition of $\ff_n$:
\be\label{2.12}
F_{f_{n-1}}(\ff_n) \ge tr_\D(f_{n-1}) 
\ee
we get
\be\label{2.13}
\|f_{n-1}+c_n\ff_n\|\ge  F_{f_{n-1}}(f_{n-1}+c_n\ff_n) 
\ee
$$
= \|f_{n-1}\| +c_n F_{f_{n-1}}(\ff_n) \ge \|f_{n-1}\| +c_ntr_\D(f_{n-1}). 
$$
Combining (\ref{2.11}) and (\ref{2.13}) we get
\be\label{2.14}
\|f_n\| = \|f_{n-1}-c_n\ff_n\| \le \|f_{n-1}\|(1+2\rho(c_n/\|f_{n-1}\|)) -c_ntr_\D(f_{n-1}). 
\ee
Using the choice of $c_m$ we get from here
\be\label{3.9}
\|f_m\|\le \|f_{m-1}\| -t(1-b)c_mr_\D(f_{m-1}). 
\ee
Thus we need to estimate from below $c_mr_\D(f_{m-1})$. It is clear that 
\be\label{3.10}
\|f_{m-1}\|_{A_1(\D)} = \|f-\sum_{j=1}^{m-1}c_j\ff_j\|_{A_1(\D)} \le \|f\|_{A_1(\D)} +\sum_{j=1}^{m-1}c_j. 
\ee
Denote $B_n:= \|f\|_{A_1(\D)}+\sum_{j=1}^nc_j$. Then by (\ref{3.10}) we have
$$
\|f_{m-1}\|_{A_1(\D)} \le B_{m-1}.
$$
Next, by Lemma 2.2 from [T3] (see also \cite{VTbook}, p.343, Lemma 6.10) we obtain
\be\label{3.11}
r_\D(f_{m-1}) = \sup_{g\in \D} F_{f_{m-1}}(g) = \sup_{\ff \in A_1(\D)} F_{f_{m-1}}(\ff) 
\ee
$$
\ge \|f_{m-1}\|_{A_1(\D)}^{-1}F_{f_{m-1}}(f_{m-1}) \ge \|f_{m-1}\|/B_{m-1}.
$$
Substituting (\ref{3.11}) into (\ref{3.9}) we get
\be\label{3.12}
\|f_m\| \le \|f_{m-1}\|(1-t(1-b)c_m/B_{m-1}). 
\ee
From the definition of $B_m$ we find
$$
B_m = B_{m-1} +c_m = B_{m-1}(1+c_m/B_{m-1}).
$$
Using the inequality
$$
(1+x)^\alpha \le 1+\alpha x, \quad 0\le \alpha\le 1, \quad x\ge 0,
$$
we obtain
\be\label{3.13}
B_m^{t(1-b)}  \le  B_{m-1}^{t(1-b)}(1+t(1-b)c_m/B_{m-1}). 
\ee
Multiplying (\ref{3.12}) and (\ref{3.13})  we get
\be\label{3.14}
\|f_m\|B_m^{t(1-b)}  \le \|f_{m-1}\| B_{m-1}^{t(1-b)} \le \cdots \le \|f\| \|f\|_{A_1(\D)}^{t(1-b)}. 
\ee
The function $\mu(u)/u = \ga u^{q-1}$ is increasing on $[0,\infty)$. Therefore the $c_m$ from (\ref{3.2}) is greater than or equal to $c_m'$ from (see (\ref{3.11}))
\be\label{3.15}
\ga \|f_{m-1}\|(c_m'/\|f_{m-1}\|)^q = \frac{tb}{2}c_m'\|f_{m-1}\|/B_{m-1}, 
\ee
\be\label{3.16}
c_m' = \left(\frac{tb}{2\ga}\right)^{\frac{1}{q-1}}\frac{\|f_{m-1}\|^{\frac{q}{q-1}}}{B_{m-1}^{\frac{1}{q-1}}}. 
\ee
Using notations
$$
p:=\frac{q}{q-1},\qquad c := (1-b)\left(\frac{b}{2\ga}\right)^{\frac{1}{q-1}},
$$
we get from (\ref{3.9}), (\ref{3.11}), (\ref{3.16})
\be\label{3.17}
\|f_m\|\le \|f_{m-1}\| \left(1-ct^p\frac{\|f_{m-1}\|^p}{B_{m-1}^p}\right). 
\ee
Noting that $B_m\ge B_{m-1}$ we derive from (\ref{3.17}) that
\be\label{3.18}
(\|f_m\|/B_m)^p \le (\|f_{m-1}\|/B_{m-1})^p (1-ct^p(\|f_{m-1}\|/B_{m-1})^p). 
\ee
Taking into account that $\|f\|\le \|f\|_{A_1(\D)}$ we obtain from (\ref{3.18}) by Lemma \ref{HL1} with $C_1=1$, $C_2=ct^p$ 
\be\label{3.19}
(\|f_m\|/B_m)^p \le (1+mc t^p)^{-1}. 
\ee
Combining (\ref{3.14}) and (\ref{3.19}) we get 
$$
\|f_m\|\le \|f\|^{1-\al_0}\|f\|_{A_1(\D)}^{\al_0}(1+mc t^p)^{-\al_0/p}, \quad p:= \frac{q}{q-1},\quad \al_0 := \frac{t(1-b)}{1+t(1-b)}.
$$
This completes the proof of Theorem \ref{BT1} for $\al=\al_0$. The case $\al <\al_0$ follows from 
the case $\al=\al_0$ and the corresponding analog of Lemma \ref{HL2}.
\end{proof} 
 
 Let us discuss an application of Theorem \ref{BT1} in the case of a Hilbert space. It is well known and easy to check that for a Hilbert space $H$ one has
$$
\rho(u) \le (1+u^2)^{1/2}-1\le u^2/2.
$$
Let us figure out how the DGA$(t,b,u^2/2)$ works in a Hilbert space. Consider the $m$th step of it. Let $\ff_m\in\D$ be from (\ref{3.1}) with $t_m=t$ (we assume existence in case $t=1$). Then it is clear that for $\ff_m$ we have
 $$
\<f_{m-1},\ff_m\> \ge t \|f_{m-1}\|r_\D(f_{m-1}) = t\sup_{g\in\D} \<f_{m-1},g\>.
$$
The WGA($t,1$) would use $\ff_m$ with the coefficient $\<f_{m-1},\ff_m\>$ at this step. 
The DGA$(t,b,u^2/2)^*$ alike WGA($t,b$) uses the same $\ff_m$ and only a fraction of $\<f_{m-1},\ff_m\>$:
\be\label{3.22}
c_m = b\|f_{m-1}\|F_{f_{m-1}}(\ff_m) = b\<f_{m-1},\ff_m\>. 
\ee
Thus the choice $b=1$ in (\ref{3.22}) corresponds to the WGA. However, it is clear from the above considerations that our technique, designed for general Banach spaces, does not work in the case $b=1$.  By Theorem \ref{BT1} with $\mu(u)=u^2/2$  the DGA$(t,b,\mu)$ and DGA$(t,b,\mu)^*$ provide the following error estimate
\be\label{3.20}
\|f_m\|\le \|f\|^{1-\al}\|f\|_{A_1(\D)}^\al (1+mc t^2)^{-\al/2}, \quad\al\le  \al_0 := \frac{t(1-b)}{1+t(1-b}.
\ee
Note that the inequality (\ref{3.20}) is similar to the corresponding inequality, which follows from Theorem \ref{T4.2a}, for $\al \le \al_1:=  \frac{t(2-b)}{2+t(2-b)}$. It is easy to check that $\al_0 < \al_1$, which means that Theorem \ref{T4.2a} gives a stronger result than the corresponding corollary of Theorem \ref{BT1}. 
 
{\bf A remark on lower bounds.} In Section \ref{H} we obtained the lower bounds, which are sharp in both 
parameters $m$ and $b$. Clearly, the most important parameter is $m$. Here we obtain the lower bounds in $m$, which apply to any algorithm providing $m$-term approximation after $m$ iterations. 
Recall the definition of the concept of $m$-term approximation with respect to a given dictionary $\D$.  
Given an integer $ m\in\N$, we  
denote by $\Sigma_m(\D)$  the set of all $m$-term approximants with respect to $\D$:
$$
\Sigma_m(\D):= \left\{h\in X \,:\,   h=\sum_{i=1}^m c_ig_i,\quad g_i\in \D,\, c_i\in \R,\, i=1,\dots,m\right\}.
$$

Define for a Banach space $X$
$$
\sigma_m(f,\D)_X := \inf_{h\in\Sigma_m(\D)}\|f-h\|_X
$$
to be the best $m$-term approximation of  $f\in X$   in the $X$-norm  with respect to $\D$.    

 Let $1< q\le 2$. Consider $X=\ell_q$. It is known (\cite{LT}, p.67) that $\ell_q$, $1<q\le 2$, is a uniformly smooth Banach space with modulus of smoothness $\rho(u)$ of power type $q$: $\rho(u)\le \gamma u^q$. Choose $\D:= \cE$ as a symmetrized standard basis $\{\pm e_j\}_{j=1}^\infty$, $e_j := (0,\dots,0,1,0,\dots)$, for $\ell_q$. For a given $m\in \N$ set 
 $$
 f:= \sum_{i=1}^{2m} e_i.
$$
Then the following relations are obvious
$$
\|f\|_{\ell_q} = (2m)^{1/q},\quad \|f\|_{A_1(\cE)} = 2m,\quad \sigma_m(f,\cE)_{\ell_q} = m^{1/q}.
$$
Therefore, for any $\al \in [0,1]$
\be\label{B25}
\frac{\sigma_m(f,\cE)_{\ell_q}}{\|f\|_{\ell_q}^{1-\al}\|f\|_{A_1(\cE)}^\al } \ge \frac{1}{2} m^{-\al/p},\quad p:= \frac{q}{q-1}.
\ee
This means that the upper bounds provided by Theorem \ref{BT1} are sharp in $m$ for any fixed parameters $t$ and $b$. More precisely, for every $q\in (1,2]$ there exist a Banach space $X$ with 
$\rho(u,X) \le \gamma u^q$ and a dictionary $\D\subset X$ such that the following inequality holds
\be\label{B26}
\frac{\sigma_m(f,\D)_{X}}{\|f\|_{X}^{1-\al}\|f\|_{A_1(\D)}^\al } \ge \frac{1}{2} m^{-\al/p},\quad p:= \frac{q}{q-1}.
\ee
Inequality (\ref{B26}) follows directly from (\ref{B25}) with $X=\ell_q$ and $\D=\cE$. 

\begin{Remark}\label{BR2} Inequality (\ref{B26}) gives the lower bound for the best $m$-term approximation. It is known (see \cite{VTbook}, Ch.6) that there are greedy type algorithms, for instance the Weak Chebyshev Greedy Algorithm and the Weak Greedy Algorithm with Free Relaxation with the weakness parameter $t\in (0,1]$, which provide the following rate of convergence for $f\in X$ with $\rho(u,X) \le \gamma u^q$, $1<q\le 2$,
$$
\|f_m\|_X \le C(q,\gamma)(1+mt^p)^{-1/p}\|f\|_{A_1(\D)}, \quad p:=\frac{q}{q-1}.
$$
This means that the upper bound in (\ref{B26}) (in the sense of order) can be realized by a greedy type 
algorithm.
\end{Remark}

Note that for specific 
$X$ and $\D$ the inequality (\ref{B26}) may be improved. We illustrate it on the example of $X=\ell_q$ with $q\in (2,\infty)$ and $\D=\cE$. Without loss of generality we can assume that 
$$
f = \sum_{i=1}^\infty c_ie_i,\quad c_1\ge c_2\ge \cdots \ge 0.
$$
Then
\be\label{B27}
\|f\|_{\ell_q} = \left(\sum_{i=1}^\infty c_i^q\right)^{1/q},\quad \|f\|_{A_1(\cE)} = \|f\|_{\ell_1} = \sum_{i=1}^\infty c_i.
\ee
We now estimate from above 
$$
\sigma_m(f,\cE)_{\ell_q} = \left(\sum_{i=m+1}^\infty c_i^q\right)^{1/q}.
$$
Our monotonicity assumption on $\{c_i\}$ imply
$$
c_m \le m^{-1/q}\|f\|_{\ell_q},\quad c_m \le m^{-1}\|f\|_{\ell_1}
$$
and, therefore, for any $\bt\in [0,1]$
\be\label{B28}
c_m \le m^{-(1/q)(1-\bt)-\bt} \|f\|_{\ell_q}^{1-\bt}\|f\|_{\ell_1}^\bt.
\ee
Setting $\al := \bt(1-1/q) +1/q$ we obtain from here with $p:= \frac{q}{q-1}$
\be\label{B29}
\sigma_m(f,\cE)_{\ell_q} \le c_m^{1-1/q}\|f\|_{\ell_1}^{1/q} \le m^{-\al/p} \|f\|_{\ell_q}^{1-\al}\|f\|_{\ell_1}^\al.
\ee
Therefore, for any $\al \in [0,1]$ we have for all $f\in \ell_q$
\be\label{B30}
\frac{\sigma_m(f,\cE)_{\ell_q}}{\|f\|_{\ell_q}^{1-\al}\|f\|_{A_1(\cE)}^\al } \le  m^{-\al/p},\quad p:= \frac{q}{q-1}.
\ee
Note that it is known (see (\ref{1.3}) above) that the space $\ell_q$ with $q\in [2,\infty)$ has the modulus 
of smoothness of the power type $2$. Thus, Theorem \ref{BT1} gives an analog of (\ref{B30}) with a weaker rate of decay $m^{-\al/2}$ than $m^{-\al/p}$ in (\ref{B30}). 

We briefly discuss another example of the $L_q$ type spaces with $q\in (2,\infty)$. Consider the space $L_q(0,2\pi)$ of real $2\pi$-periodic functions and take $\D=\Tr$ to be the real trigonometric system. 
For a given $m\in \N$ set
$$
f:= \sum_{k=1}^{2m} \cos(2^kx).
$$
Then it is well known that
$$
\|f\|_{L_q} \le C(q)(2m)^{1/2},\quad \|f\|_{A_1(\Tr)} = 2m.
$$
Also,
$$
\sigma_m(f,\Tr)_{L_q} \ge \sigma_m(f,\Tr)_{L_2} \ge Cm^{1/2}.
$$
Therefore, for any $\al \in [0,1]$ we obtain 
\be\label{B31}
\frac{\sigma_m(f,\Tr)_{L_q}}{\|f\|_{L_q}^{1-\al}\|f\|_{A_1(\Tr)}^\al } \ge C'(q) m^{-\al/2}.
 \ee

{\bf Acknowledgements.} This research was supported by the Russian Science Foundation (project No. 23-71-30001)
at the Lomonosov Moscow State University.


\begin{thebibliography}{9999}

\bibitem{DGDS} M. Donahue, L. Gurvits,  C. Darken, E. Sontag,
Rate of convex approximation in non-Hilbert spaces, Constr.
Approx., {\bf 13} (1997),187--220.

\bibitem{DT} R.A. DeVore and V.N. Temlyakov, Some remarks
on Greedy Algorithms, Advances in Comp. Math., {\bf 5} (1996),
 173--187.
 
 \bibitem{LT} J. Lindenstrauss and L. Tzafriri, Classical Banach Spaces I,
 Springer-Verlag,  Berlin,  1977. 

 \bibitem{Liv} E.D. Livshitz, On lower estimates of rate of convergence of greedy algorithms,  Izv. RAN, Ser. Matem., {\bf 73}  (2009), 125--144.


\bibitem{VT111} V.N. Temlyakov, Greedy Expansions in Banach Spaces, Advances in Comput. Math. {\bf 26} (2007), 431--449.

\bibitem{VTbook} V.N. Temlyakov, Greedy approximation, Cambridge University
Press, 2011.

\end{thebibliography}
\end{document}